\documentclass[]{interact}
\usepackage{tablefootnote}
\usepackage[linktoc=page]{hyperref}
\hypersetup{linkcolor  = blue}
\hypersetup{citecolor=blue} 
\hypersetup{colorlinks=true}
\hypersetup{urlcolor=cyan}
\usepackage{color, colortbl}
\usepackage{tabularx}
\usepackage{graphicx, amsmath, amsthm, amssymb, mathrsfs, amscd}
\usepackage{caption}
\usepackage{enumerate}
\usepackage{algorithm}
\usepackage[noend]{algcompatible}
\usepackage{url}
\usepackage{tikz}
\usepackage[toc,page]{appendix}
\usepackage[norule,symbol,perpage]{footmisc}
\usepackage{footnote}
\usepackage{color}
\usepackage{color, colortbl}
\usepackage{tablefootnote}
\usepackage{footnote}
\usepackage{epstopdf}
\usepackage[caption=false]{subfig}
\usepackage[numbers,sort&compress]{natbib}
\bibpunct[, ]{[}{]}{,}{n}{,}{,}
\makeatletter
\def\NAT@def@citea{\def\@citea{\NAT@separator}}
\makeatother
\theoremstyle{plain}
\newtheorem{theorem}{Theorem}[section]

\theoremstyle{definition}
\newtheorem{definition}[theorem]{Definition}
\newtheorem{example}[theorem]{Example}
\theoremstyle{remark}
\newtheorem{remark}{Remark}

\usepackage{multirow}
\usepackage{caption}
\usepackage{enumerate}
\usepackage{algorithm}
\usepackage[noend]{algcompatible}
\usepackage{url}
\usepackage{tikz}
\usepackage[toc,page]{appendix}
\usepackage{footnote}

\usepackage{color}
\usepackage{color, colortbl}
\usepackage{tablefootnote}
\usepackage{footnote}
\theoremstyle{definition}

\theoremstyle{definition}

\theoremstyle{definition}

\theoremstyle{definition}
\theoremstyle{definition}
\makeatletter
\def\cleardoublepage{\clearpage\if@twoside \ifodd\c@page\else
  \hbox{}
  \vspace*{\fill}
    \vspace{\fill}
  \if@twocolumn\hbox{}\newpage\fi\fi\fi}
\makeatother

\begin{document}
\title{VALENCY OF CERTAIN COMPLEX-VALUED FUNCTIONS}
\author{
\bigskip \name{OLUMA ARARSO ALEMU}
 \affil{Department of Mathematics, Addis Ababa University, Addis Ababa, Ethiopia. \\  Email: oluma.ararso@aau.edu.et}}
\maketitle
\begin{abstract}
The valence of a function $f$ at a point $z_0$ is the number of distinct,finite solutions to $f(z) = z_0.$ In this paper we bound the valence of complex-valued harmonic polynomials in the plane for some special harmonic polynomials of the form  $f(z) = p(z)\overline{q(z)}$, where $p(z)$ is an analytic polynomial of degree $n$ and $q(z)$ is an analytic polynomial of degree $m,$ and  $q(z) \neq \alpha p(z)$ for some constant $\alpha.$ Using techinques of complex dynamics used in the work Sheil-Small and Wilmshurst on the valence of harmonic polynomial, we prove that the harmonic polynomial $f(z) = p(z)\overline{q(z)}$ has the valency of $m+n.$
\end{abstract}

\section{\textcolor{blue}{INTRODUCTION}}$\label{I}$
 An upper bound on the total number of roots of analytic function can be known due to the result of the fundamental theorem of algebra, which stipulates that every non-zero, single-variable, analytic polynomial of degree $n$ with complex coefficients has  \bigskip exactly $n$ complex roots, counted with multiplicity. \\
  The argument principle for harmonic functions can be formulated as a direct generalization of the classical result for analytic functions, see Duren  \textit{et al} \cite{duren1996argument}.
 The winding number of the image curve $f(C)$ about the origin, $\frac{1}{2\pi}\Delta _C~ arg f(z)$  equals the total number of zeros of $f$ in $\mathbb{D},$ counted according to multiplicity where $\mathbb{D}$ is a plane domain bounded by a rectifiable Jordan  curve C, oriented in the counterclockwise direction \bigskip and $f(z) \neq 0$  on   C is  analytic in $\mathbb{D}$ and  continuous in $\overline{\mathbb{D}}.$ \\
 The location and number of the zeros of analytic polynomials as well as complex-valued harmonic polynomials has been studied by many researchers. 
After a basic paper by Clunie and Sheil Small in 1984 \cite{clunie1984harmonic}, the theory of Harmonic univalent and multivalent functions attracted attention of complex analysts. Since then the theory has rapidly developed and became an active branch of complex analysis for researchers. One area of investigation that has recently become of interest is the number and location of zeros of complex-valued harmonic polynomials. Clunie and Sheil-Small \cite{clunie1984harmonic} introduced the family of complex-valued harmonic functions $f = u + iv$  defined in the unit disk \bigskip $\mathbb{D} = \{z : |z| < 1\},$ where $u$ and $v$ are real harmonic in $\mathbb{D}.$\\
 Lewy's Theorem stipulates that if $f$ is a complex valued harmonic function that is locally univalent in a domain $\mathbb{D} \subset \mathbb{C},$ then its Jacobian $,J_f(z),$ never vanish for all $z \in \mathbb{D}.$ For more information we can refer to \cite{lewy1936non}.  As an immediate consequence of Lewy's Theorem, a complex valued harmonic function $f(z)=h(z)+\overline{g(z)}$ is locally univalent and sense-preserving if and only if $h'(z) \neq 0$ and $|\omega (z)|<1,$ where $\omega (z)$ is a  \bigskip dilatation  function of $f$ defined by $\omega (z) = \frac{g'(z)}{h'(z)}.$ \\ 
Determining the number of zeros of complex-valued harmonic polynomial is an interesting research area in complex analysis. Suppose $f$ is a polynomial function of degree $n$ in two variables and let $\mathcal{Z}_f$ denote the number of zeros of $f,$ that is, number of points $z \in \mathbb{C}$ satisfying $f(z) = 0.$ Then for $n > m,$ we have $n \leq \mathcal{Z}_f \leq n^2.$ The lower bound is based on the generalized argument principle and is sharp for each $m$ and $n.$ The upper bound follows from applying Bezout’s theorem to the real and imaginary parts of $f(z) = 0$ after noticing that the zeros are isolated, which was shown by Wilmshurst \cite{wilmshurst1998valence} and this upper bound is sharp in general. For instance as illustrated by Wilmshurst in the same paper,  $\left(\frac{1}{i}\right)^n Q\left(iz+\frac{1}{2} \right) $ is a polynomial with $n^2$  zeros where $Q(z)=z^n+(z-1)^n+i\overline{z}^n -i(\overline{z}-1)^n.$ But it is natural to ask whether or not it can  \bigskip be improved for some interesting special classes of polynomials.\\
Wilmshurst \cite{hauenstein2015experiments} considered the cases of complex-valued harmonic polynomials of the form  
$f(z)=p(z)+\overline{q(z)}.$ If $\mathrm{deg}p=n>m=\mathrm{deg}q,$ then as to the question of improving the bound $\mathcal{Z}_f \leq n^2$ given additional information, Wilmshurst made the conjecture $\mathcal{Z}_f \leq 3n-2+m(m-1).$ This conjecture is stated in \cite{wilmshurst1998valence}. It is also among the list of open problems in\cite{bshouty2010problems}. For $m=n-1$ the upper bound follows from Wilmshurst's theorem  and examples were also given in \cite{wilmshurst1998valence} showing that this bound is sharp. For $m = 1,$ the upper bound was proved by D.Khavinson and G.Swiatek \cite{khavisonmaximal}, and bound was also sharp.  For $m = n- 3$, the conjectured bound is $3n - 2+ m(m- 1) = n^ 2 - 4n + 10.$ \\
Using techniques of complex dynamics, it is  still possible to improve for some special harmonic polynomials of the form $f(z) = p(z)\overline{q(z)}$, where $p(z)$ is an analytic polynomial of degree $n$ and $q(z)$ is an analytic polynomial of degree $m,$ and  $q(z) \neq \alpha p(z)$ for some constant $\alpha.$ The valency of this type of function is an open problem in \cite{bshouty2010problems} and we solve it here.
\begin{definition} \cite{abdulhadi2012univalent}
Let $\mathbb{D}$ be an open unit disk in the complex plane $\mathbb{C}.$ Denote by $\mathcal{H}(\mathbb{D})$ the linear space of all analytic  functions in  $\mathbb{D}$, and let  $\mathcal{B}(\mathbb{D})$ be the set of all functions $\omega(z) \in \mathcal{H}(\mathbb{D})$ satisfying $|\omega(z)| < 1, z \in \mathbb{D}.$ A non-constant function $f$ is said to be \textbf{logharmonic} in  $\mathbb{D}$ if $f$ is the solution of the nonlinear elliptic differential equation
$$ 
\overline{f_{\overline{z}}(z)} = \omega(z)\left( \frac{\overline{f(z)}}{f(z)} \right) f_z(z)
$$ where $\omega \in \mathcal{B}(\mathbb{D})$ and the function $\omega(z)$ is called the second dilatation of $f.$
\end{definition}
\begin{remark} For $f(z)$ as in the above definition, 
\begin{itemize}
\item[i.] The Jacobian $,J_f$ of $f$ is given by
\begin{equation}
J_f = |f_z|^2-|f_{\overline{z}}|^2 = |f_z|(1 - |\omega |^2),
\end{equation}
which is positive, and therefore, every non-constant log-harmonic mapping is sense-preserving
and open in $\mathbb{D}.$
\item[ii.]  If $f$ is a non-vanishing log-harmonic mapping in $\mathbb{D},$ then $f$ can be expressed as
$$
f(z) = h(z)\overline{g(z)},
$$
where $h(z)$ and $g(z)$ are non-vanishing analytic functions in $\mathbb{D}.$ 
\item[iii.] if $f$ vanishes at $z = 0$ but is not identically zero, then such a mapping $f$ admits the following
representation
$$
f(z) = z|z|^{2\beta}h(z)\overline{g(z)},
$$
where $Re \beta  > \frac{-1}{2}, h, g \in \mathcal{H}(\mathbb{D}), h(0) \neq 0$ and $g(0) = 1$, see \cite{abdulhadi1988univalent}.
\end{itemize}
\end{remark}
\begin{remark}[The previous Conjecture]\cite{hauenstein2015experiments}
 Let $f(z) = p(z) + \overline{q(z)},$ where $\mathrm{deg}q(z) = m$ and $\mathrm{deg}p(z) = n,$ satisfy
 $f(z) \rightarrow \infty $ as $z \rightarrow \infty.$ By conjugating $f(z),$ if necessary, we may assume that $n > m.$ For $m = n$ and $m = n-1$ it has been shown that the valence bound $n^2$ to be sharp, but for $m < n- 1$ it would be surprising if it were still possible to obtain $n^2$ valence.The correct bound is $m(m - 1) + 3n - 2$ for $1 < m < n - 1.$ For $m = 1$ this becomes $\overline{z} + h(z)$ is at most $3n - 2$ valent, or equivalently, that the number of fixed points of the conjugate of an analytic polynomial of degree $n$ is at most $3n - 2.$
 \end{remark}
 
\begin{theorem} \cite{de2012fundamental} If $p(z)$ is a non constant complex polynomial, then $p$ has a zero in the complex field $\mathbb{C}.$
\end{theorem}
As an immediate consequence of the Fundamental Theorem of Algebra, if $p(z)$ is a polynomial and $a_1,a_2,...,a_m$ are its zeros with $a_j$ having multiplicity $k_j,$ then 
$$
p(z)= c(z-a_1)^{k_1}(z-a_2)^{k_2}...(z-a_m)^{k_m}
$$
 for some constant $c$ and $k_1+k_2+...+k_m$ is the degree of $p.$ 
\begin{theorem} \cite{kirwan1992complex} $\label{II0}$
 Let $f$ and $g$ be relatively prime polynomials in the real variables $x$ and $y$ with real coefficients, and let $\mathrm{deg} h = n$ and $\mathrm{deg}g = m.$ Then the two algebraic curves $f(x, y) = 0$ and $g(x, y) = 0$ have at most $mn$ points in common.
  \end{theorem}
 A Bezout's theorem  is very useful to bound the number of zeros of complex-valued harmonic polynomials.

  \begin{theorem} \cite{khavinson2018zeros} $\label{II00}$
 For a harmonic polynomial $f(z) = p(z) + \overline{q(z)}$ with real coefficients, the equation $f(z) = 0$ has at most $ n^2 - n$ solutions that satisfy $(Re z)(Imz) \neq 0$ where $\mathrm{deg}p =n>m=\mathrm{deg}q.$
  \end{theorem}
If we consider $p(z) = z^n + (z - 1)^n$ and $q(z) = z^n - (z - 1)^n,$ then $f(z) = p(z) + \overline{q(z)}$ has $n^2$ number of roots including the root at $0$ with the multiplicity $n.$ In fact, this is the polynomial that Wilmshurst used (with a slight perturbation to split the multiple root at the origin) to show that the maximal bound $n^2$ is sharp. This is the reason why in Theorem  $\ref{II00}$ we only consider roots of the coordinate axes and it yields that the harmonic polynomial with real coefficients and with the maximal $n^2$ number of roots should have at least $n$ roots on the coordinate axes.
  \begin{theorem}\cite{wilmshurst1998valence} $\label{II000}$
  Let $f$ be a function harmonic in the (entire) complex plane. If $liminf_{z \rightarrow \infty}|f| > 0,$ then $f$ has finitely many zeros.
  \end{theorem}
\begin{theorem} \cite{wilmshurst1998valence}$\label{II1}$
If $f(z)= p(z)+\overline{q(z)} $ is a harmonic polynomial such that $\mathrm{deg}p=n>m=\mathrm{deg}q$ and $\lim_{z \rightarrow \infty}f(z) = \infty,$ then $f(z)$ has at most $n^2$ zeros.
\end{theorem}
 The proof of this result can also readily follows from Bezout's theorem.\bigskip \\
 The following result is stated and peoved in \cite[Theorem~1]{khavinson2003number}.
\begin{theorem}$\label{II2}$
Let $q(z)= z-\overline{p(z)},$ where $p(z)$ is an analytic polynomial with condition $\mathrm{deg}p=n>1.$  Then $\mathcal{Z}_q \leq 3n-2.$
\end{theorem}
 \section{\textcolor{blue}{RESULT}} $\label{III}$ 
 Our main result is the following.
\begin{theorem}$\label{III1}$
Let $f(z) = p(z)\overline{q(z)}$, where $p(z)$ is an analytic polynomial of degree $n$ and $q(z)$ is an analytic polynomial of degree $m,$ and let $q(z) \neq \alpha p(z)$ for some constant $\alpha.$ Then $\mathcal{Z}_f \leq m+n.$ 
\end{theorem}
\begin{proof}
Since $p(z)$ is an analytic polynomial of degree $n$ and $q(z)$ is an analytic polynomial of degree $m,$ we have a power series representations:
$$
p(z)= \sum _{k=0}^n a_kz^k ~~~ and ~~~ q(z)=\sum _{j=0}^m b_jz^j. 
$$
 Then by the Fundamental Theorem of Algebra, we can decompose $p(z)$ and $q(z)$ as 
$$
 p(z)=\beta (z-a_1)(z-a_2)...(z-a_n) ~~~ and ~~~ q(z)= \gamma (z-b_1)(z-b_2)...(z-b_m). 
$$
This directly implies that $p(z)$ has at most $n$ distinct complex roots and $q(z)$ has at most $m$ distinct complex roots. It then follows from 
$$
f(z)=\underbrace{\underbrace{(\beta \overline{ \gamma)} (z-a_1)(z-a_2)...(z-a_n)}_{\leq n~ distinct~zeros} \underbrace{\overline{(z-b_1)(z-b_2)...(z-b_m)}}_{\leq m~ distinct~zeros}.}_{\leq n+m~distinct~zeros.}
$$
It is not surprising that $q(z)$ and $\overline{q(z)}$ have the same number of zeros. Therefore, $\mathcal{Z}_f \leq m+n.$ 
\end{proof}
\section{\textcolor{blue}{SHARPNESS OF THE BOUND}}
\begin{example}
$f(z)=|z|^2z+\overline{z}$ is a complex-valued function and can be rewritten as $f(z)=(z^2+1)(\overline{z}).$ In our case, $p(z)=z^2+1$ and $q(z)=z.$ We have here, $z=i, z=-i$ and $z=0$ are the three distinct zeros of $f.$ Therefore, the valency $\mathcal{Z}_f \leq m+n$ is sharp.
\end{example}
 \section{\textcolor{blue}{CONCLUSION}}$\label{V}$
In this paper we have determined the maximum number of the  zeros of harmonic polynomial of the form $f(z) = p(z)\overline{q(z)}$, where $p(z)$ is an analytic polynomial of degree $n$ and $q(z)$ is an analytic polynomial of degree $m,$ and  $q(z) \neq \alpha p(z)$ for some constant $\alpha.$ The result shows that the valence of  $f(z)$ is $m+n.$
\section*{\textcolor{blue}{ACKNOWLEDGMENTS}}
I would like to thank Addis Ababa University, Simons Foundation, and ISP(International Science Program) at department level for providing us opportunities and financial support.
\section*{\textcolor{blue}{DECLARATION OF INTEREST}}
The author declare that there are no conflicts of interest regarding the publication of this paper.


\begin{thebibliography}{99}
\bibitem{abdulhadi2012univalent}Abdulhadi, Z., \& Ali, R. M. (2012, January). Univalent logharmonic mappings in the plane. In Abstract and Applied Analysis (Vol. 2012). Hindawi.
\bibitem{abdulhadi1988univalent}Abdulhadi, Z., \& Bshouty, D. (1988). Univalent functions in $H \overline H (z)$. Transactions of the American Mathematical Society, 305(2), 841-849.
\bibitem{alemu2022zeros}Alemu, O. A., \& Geleta, H. L. (2022). Zeros of a two-parameter family of harmonic quadrinomials. SINET: Ethiopian Journal of Science, 45(1), 105-114.
\bibitem{ararso2022image}Ararso Alemu, O., \& Legesse Geleta, H. (2022). The Image of Critical Circle and Zero-free Curve for Quadrinomials. arXiv e-prints, arXiv-2211.
\bibitem{ararso2021zeros} Ararso Alemu, O., \& Legesse Geleta, H. (2021). Zeros of a Two-parameter Family of Harmonic Quadrinomials. arXiv e-prints, arXiv-2106.
\bibitem{bshouty2010problems}Bshouty, D., \& Lyzzaik, A. (2010, August). Problems and conjectures in planar harmonic mappings. In Proceedings of the ICM2010 Satellite Conference International Workshop on Harmonic and Quasiconformal Mappings, Editors: D. Minda, S. Ponnusamy, and N. Shanmugalingam, J. Analysis (Vol. 18, pp. 69-81).
\bibitem{clunie1984harmonic}Clunie, J., and Sheil-Small, T. (1984). Harmonic univalent functions. Ann. Acad. Sci. Fenn.
\bibitem{khavisonmaximal}D. Khavinson, G. Swiatek, On a maximal number of zeros of certain harmonic polynomials Proc. AMS, 131 (2003), 409-414.
\bibitem{de2012fundamental}de Oliveira, O. R. B. (2012). The fundamental theorem of algebra: from the four basic operations. The American Mathematical Monthly, 119(9), 753-758.
\bibitem{hauenstein2015experiments}Hauenstein, J. D., Lerario, A., Lundberg, E., \& Mehta, D. (2015). Experiments on the zeros of harmonic polynomials using certified counting. Experimental Mathematics, 24(2), 133-141.
\bibitem{duren1996argument}Duren, P., Hengartner, W., and Laugesen, R. S. (1996). The argument principle for harmonic functions. The American mathematical monthly, 103(5), 411-415.
\bibitem{khavinson2018zeros}Khavinson, D., Lee, S. Y., \& Saez, A. (2018). Zeros of harmonic polynomials, critical lemniscates, and caustics. Complex Analysis and its Synergies, 4(1), 1-20.
\bibitem{khavinson2003number} Khavinson, D., \& Ĺšwiatek, G. (2003). On the number of zeros of certain harmonic polynomials. Proceedings of the American Mathematical Society, 131(2), 409-414.
\bibitem{kirwan1992complex}Kirwan, F. C., \& Kirwan, F. (1992). Complex algebraic curves (Vol. 23). Cambridge University Press
\bibitem{lee2015remarks}Lee, S. Y., Lerario, A., \& Lundberg, E. (2015). Remarks on Wilmshurst's theorem. Indiana University Mathematics Journal, 1153-1167.
\bibitem{lewy1936non}Lewy, H. (1936). On the non-vanishing of the Jacobian in certain one-to one mappings. Bulletin of the American Mathematical Society, 42(10), 689-692
\bibitem{wilmshurst1998valence}Wilmshurst, A. S. (1998). The valence of harmonic polynomials. Proceedings of the American Mathematical Society, 2077-2081.
\end{thebibliography}
\end{document}